\documentclass[a4paper]{amsart}

\usepackage{tikz} 
\usetikzlibrary{cd, positioning, intersections, calc, arrows.meta, math}

\usepackage{amsmath}
\usepackage{amssymb}
\usepackage{amsfonts}
\usepackage{amsthm}
\usepackage{latexsym}

\usepackage{amsaddr}

\usepackage{stmaryrd}

\usepackage{comment} 

\numberwithin{equation}{section}

\newtheorem{definition}{Definition}[section]
\newtheorem{theorem}[definition]{Theorem}
\newtheorem{lemma}[definition]{Lemma}
\newtheorem{corollary}[definition]{Corollary}

\newcommand{\cformulas}{\Phi_{\mathrm{CKL}}}
\newcommand{\icformulas}{\mathrm{I}\Phi_{\mathrm{CKL}}}

\newcommand{\psmh}{\mathbf{PS}_{\mathrm{MH}}}
\newcommand{\psomega}{\mathbf{PS}_{\omega}}
\newcommand{\ipsmh}{\mathbf{IPS}_{\mathrm{MH}}}
\newcommand{\ipsomega}{\mathbf{IPS}_{\omega}}
\newcommand{\psnon}{\mathbf{PS}}

\newcommand{\lckl}{\mathbf{CKL}}
\newcommand{\fckl}{\mathbf{Frm}_{\mathrm{CKL}}}
\newcommand{\ackl}{\mathbf{Alg}_{\mathrm{CKL}}}
\newcommand{\amh}{\mathbf{Alg}_{\mathrm{MH}}}
\newcommand{\freeomega}{\mathrm{F}_{\mathrm{CKL}}}
\newcommand{\freemh}{\mathrm{F}_{\mathrm{MH}}}
\newcommand{\linmh}{\mathrm{LT}_{\psmh}}
\newcommand{\linomega}{\mathrm{LT}_{\psomega}}
\newcommand{\linnon}{\mathrm{LT}_{\psnon}}

\newcommand{\modalop}{\mathcal{I}}
\newcommand{\mdl}{\Box}

\newcommand{\ko}{\mathsf{K}}
\newcommand{\eko}{\mathsf{E}}
\newcommand{\cko}{\mathsf{C}}
\newcommand{\ako}{\mathop{\mathrm{K}}\nolimits}
\newcommand{\aeko}{\mathop{\mathrm{E}}\nolimits}
\newcommand{\acko}{\mathop{\mathrm{C}}\nolimits}

\newcommand{\deflogic}{\mathsf{For}}
\newcommand{\defframe}{\mathsf{Frm}}
\newcommand{\defalgebra}{\mathsf{Alg}}

\newcommand{\ckl}{\mathbf{CKL}}

\newcommand{\logic}[1]{\mathbf{#1}}

\newcommand{\iand}{\bigwedge}
\newcommand{\ior}{\bigvee}

\newcommand{\thn}{\ \Rightarrow\ }
\newcommand{\eq}{\ \Leftrightarrow\ }

\newcommand{\vl}{\models}

\newcommand{\gm}{\Gamma}

\newcommand{\barcan}{\forall x\Box\phi\supset\Box\forall x\phi}

\newcommand{\power}{\mathcal{P}}

\newcommand{\propvar}{\mathsf{Prop}}

\newcommand{\pscom}{\mbox{$\rightarrow$}}

\begin{document}

\title{
Models for the common knowledge logic
}
\author{Yoshihito Tanaka}
\address{Kyushu Sangyo University, Fukuoka, Japan}
\email{ytanaka@ip.kyusan-u.ac.jp}

\date{}

\begin{abstract}
We discuss models of the common knowledge logic,  
a multi-modal logic with 
modal operators $\ko_{i}$ ($i\in\modalop$) and $\cko$.  
The intended meaning of
$\cko\phi$ is 
$\phi\land\eko\phi\land\eko^{2}\phi\cdots$, 
where $\eko\phi=\iand_{i\in\modalop}\ko_{i}\phi$. 
A Kripke frame for this, called a CKL-frame, is 
$\langle W,R_{\ko_{i}} (i\in\modalop), R_{\cko}\rangle$, 
where $R_{\cko}$ is the reflexive and transitive closure of 
$R_{\eko}=\bigcup_{i\in\modalop}R_{\ko_{i}}$, and an algebra for this,  
called a CKL-algebra, is 
a modal algebra with operators 
$\ako_{i}$ ($i\in\modalop$) and $\acko$,
satisfying 
$\acko x\leq\aeko\acko x$ 
and 
$\acko x=\bigsqcap_{n\in\omega}\aeko^{n}x$, 
where $\aeko x=\bigsqcap_{i\in\modalop} \ako_{i} x$. 
We show that the class of CKL-frames 
is modally definable, whereas 
the class of CKL-algebras 
is not. 
That is, 
the class of CKL-algebras is not a variety, and
there exists a modal algebra in which the common knowledge logic 
is valid, but $\acko x\not=\bigsqcap_{n\in\omega}\aeko^{n}x$. 
\end{abstract}

\maketitle


\bigskip
\noindent
{\bf Statements and Declarations}\\
Competing Interests: The author declares that he has no competing interests.

\bigskip
\noindent
{\bf MCS}\\
03B42, 03B45

\newpage

\section{Introduction}

In this paper, we discuss models of the common knowledge logic. 
The common knowledge logic is a multi-modal logic that includes the
modal operators $\ko_{i}$ ($i\in\modalop$, where 
$\modalop$ is a finite set of agents) and $\cko$ in the language.  
The intended meanings of $\ko_{i}\phi$ ($i\in\modalop$) 
and $\cko\phi$ are ``the agent $i$ knows $\phi$'' ($i\in\modalop$) and 
``$\phi$ is common knowledge 
among $\modalop$'', respectively. 
Semantically, this can be expressed as follows: 
$\cko\phi$ is true if and only if all of 
$\phi$, $\eko\phi$, $\eko^{2}\phi$, $\eko^{3}\phi,\ldots$ 
are true, 
where $\eko\phi=\iand_{i\in\modalop}\ko_{i}\phi$ 
(see, e.g., \cite{sgb94,mey-vdh95,wlt00,knk-ngs-szk-tnk}). 
\footnote{
Some studies define that 
$\cko\phi$ is true if and only if all of 
$\eko\phi$, $\eko^{2}\phi$, $\eko^{3}\phi,\ldots$ 
are true (see, e.g., \cite{hlp-mss92,bch-kzn-std10}). 
However, our research can be easily modified to apply in such cases, as well.
}
A Kripke frame that satisfies the condition is 
$\langle W,R_{\ko_{i}} (i\in\modalop), R_{\cko}\rangle$, 
where $R_{\cko}$ is the reflexive and transitive closure of 
$R_{\eko}=\bigcup_{i\in\modalop} R_{\ko_{i}}$. 
We call such Kripke frames as CKL-frames. 
An algebra that satisfies the condition is 
a modal algebra with unary operators 
$\ako_{i}$ ($i\in\modalop$) and $\acko$,
which satisfies that 
$\acko x\leq\aeko\acko x$ 
and 
$\acko x=\bigsqcap_{n\in\omega}\aeko^{n} x$, 
where $\aeko x=\bigsqcap_{i\in\modalop} \ako_{i} x$. 
We call such modal algebras as CKL-algebras. 
In this paper, we show that the class of CKL-frames 
is modally definable. 
Then, we show that the class of CKL-algebras 
does not form a variety, 
indicating that the class of CKL-algebras is not modally definable. 
Consequently, it follows that 
there exists 
a modal algebra in which the common knowledge logic 
is valid, but $\acko x\not=\bigsqcap_{n\in\omega}\aeko^{n}x$.

We use the proof systems $\psmh$, originally introduced by 
Meyer and van der Hoek \cite{mey-vdh95}, and 
$\psomega$, originally introduced by 
Kaneko-Nagashima-Suzuki-Tanaka \cite{knk-ngs-szk-tnk}, 
to discuss proof theoretic properties of the common knowledge logic. 
The system $\psmh$ includes the following three axiom schemas 
for the common knowledge 
operator:   
\begin{equation}\label{indpsmh}
\cko\phi\supset\phi,\ 
\cko\phi\supset\eko\cko\phi,\ 
\cko(\phi\supset\eko\phi)\supset(\phi\supset\cko\phi). 
\end{equation}
The system $\psomega$ includes the first two axiom schemas of (\ref{indpsmh})
and the following $\omega$-rule: 
\begin{equation*}
\dfrac{
\gamma
\supset
\mdl_{1}(
\phi_{1}\supset
\mdl_{2}(
\phi_{2}\supset
\cdots\supset
\mdl_{k}(
\phi_{k}\supset
\eko^{n}\phi
)
\cdots
)
)
\text{\rm \hspace{10pt}($n\in\omega$) }
}
{
\gamma
\supset
\mdl_{1}(
\phi_{1}\supset
\mdl_{2}(
\phi_{2}\supset
\cdots\supset
\mdl_{k}(
\phi_{k}\supset
\cko\phi
)
\cdots
)
)
}.
\end{equation*}
It is shown in \cite{mey-vdh95} and 
\cite{knk-ngs-szk-tnk} that $\psmh$ and $\psomega$ are sound and complete 
with respect to the class of CKL-frames, respectively. 
Then, it follows immediately that both systems are sound and complete 
with respect to the class of CKL-algebras. 
We write $\ckl$ for the set of formulas that are derivable in $\psmh$ and $\psomega$. 
For the knowledge operators $\ko_{i}$ ($i\in\modalop$), we only assume
the axioms for the normal modal logic $\logic{K}$. 
However, the results of the paper hold for 
models for 
$\logic{K}$, $\logic{D}$, $\logic{T}$, $\logic{4}$ $\logic{5}$ and their combinations.

Then, we prove that 
the class of CKL-frames 
is modally definable, that is, there exists a set $S$ of formulas of the common 
knowledge logic such that 
the class of CKL-frames is equal to the class of Kripke frames in which $S$ is valid. 
In fact, we show that the class of CKL-frames is equal to the class of Kripke frames 
in which $\lckl$ is valid. 
On the other hand, we show that the class of CKL-algebras is not modally definable. 
We first introduce the notion of MH-algebras, 
an algebraic counterpart of the system $\psmh$. 
Then, it is shown that 
the class of MH-algebras
is equal to 
the class of modal algebras in which $\lckl$ is valid. 
Subsequently, it is proved that 
the class of CKL-algebras is not modally definable.  
This implies the existence of a modal algebra validating the 
common knowledge logic, but does not satisfy 
$\acko x=\bigsqcap_{n\in\omega}\aeko^{n}x$. 
Additionally, we show that the free MH-algebras are CKL-algebras.

In the final section of the paper, we discuss infinitary common knowledge logic. 
The infinitary common knowledge logic is 
introduced by Kaneko-Nagashima in \cite{knk-ngs96,knk-ngs97} 
to provide a mathematical logic framework for investigations of game 
theoretical problems.
Then, it is shown by 
Kaneko-Nagashima-Suzuki-Tanaka \cite{knk-ngs-szk-tnk} 
that the infinitary extension of $\psomega$ 
is sound and complete with respect to the class of CKL-frames. 
We show that the infinitary extension of $\psmh$ is also sound and complete 
with respect to the class of CKL-frames. 
Infinitary extensions of logics are often compared with predicate extensions 
of logics. 
It is proved in \cite{knk-ngs-szk-tnk} 
that the predicate extension of $\psomega$ is sound and complete 
with respect to the class of CKL-frames. 
However, 
it is shown by Wolter \cite{wlt00} that 
the predicate extension of $\psmh$ is not complete with respect to the 
class of CKL-frames. 
Our result is one of the examples demonstrating the 
distinct properties exhibited by infinitary extensions and 
predicate extensions.

The paper is organized as follows. In Section \ref{sec:pre}, 
we recall basic definitions and fix notation. 
In Section \ref{sec:ps}, 
we give definitions of the proof systems $\psmh$ and 
$\psomega$, 
introduced by Meyer and van der Hoek \cite{mey-vdh95} and 
Kaneko-Nagashima-Suzuki-Tanaka \cite{knk-ngs-szk-tnk}, respectively. 
In Section \ref{sec:models}, 
we show that the class of CKL-frames is
modally definable, while 
the class CKL-algebras is not. 
In Section \ref{sec:inflogic}, 
we discuss the infinitary common knowledge logic.

\section{Preliminaries}\label{sec:pre}

In this section, we recall basic definitions and fix notation. 
Throughout the paper, we write $\modalop$ for a non-empty finite set of agents.

The language for the common knowledge logic consists of the 
following symbols:
\begin{enumerate} 
\item
a countable set $\propvar$ of propositional variables;

\item
$\bot$ and $\top$;

\item logical connectives: 
$\lor$, $\land$, and $\neg$;

\item
modal operators $\ko_{i}$ ($i\in\modalop$) and $\cko$.
\end{enumerate}
The set $\cformulas$ of formulas is defined in the usual way. 
We write $\phi\supset\psi$, 
$\phi\equiv\psi$,  and $\eko\phi$ to abbreviate
$\neg\phi\lor\psi$, 
$(\phi\supset\psi)\land(\psi\supset\phi)$, and $\iand_{i\in\modalop}\ko_{i}\phi$, 
respectively. 
For each $n\in\omega$, we define $\eko^{n}\phi$ by 
$\eko^{0}\phi=\phi$ and $\eko^{n+1}\phi=\eko(\eko^{n})\phi$.

Throughout this paper, we consider 
multi-modal algebras with modal operators 
$\ako_{i}\ (i\in\modalop)$ and $\acko$
as algebraic models for the common knowledge logic, 
and we call such multi-modal algebras simply as modal algebras.

\begin{definition}
An algebra 
$\langle A,\sqcup,\sqcap,-,\ako_{i}\ (i\in\modalop),\acko,0,1,\rangle$
is called a {\em modal algebra} if 
$\langle A,\sqcup,\sqcap,-,0,1,\rangle$ is a Boolean algebra and 
$\mdl 1=1$ and $\mdl(x\sqcap y)=\mdl x\sqcap\mdl y$ hold 
for each modal operator $\mdl$
and each $x$ and $y$ in $A$. 
For each $x$ in $A$, we write 
$\aeko x$ to abbreviate $\bigsqcap_{i\in\modalop} \ako_{i} x$, and  
for each $x$ and $y$ in $A$,
we write $x\pscom y$ for $-x\sqcup y$. 
A modal algebra $A$ is said to be {\em complete}, 
if for any subset $S\subseteq A$, the least upper bound 
$\bigsqcup S$ and 
the greatest 
lower bound $\bigsqcap S$ exist in $A$, 
and said to be {\em completely multiplicative}, 
if for any subset $S\subseteq A$, 
\begin{equation}\label{abarcan}
\mdl\bigsqcap S=\bigsqcap_{s\in S}\mdl s
\end{equation}
holds for any modal operator $\mdl$. 
A modal algebra $A$ is said to be {\em $\omega_{1}$-complete}, if
$\bigsqcup S$ and 
$\bigsqcap S$ exist in $A$ for any countable subset $S\subseteq A$ and 
said to be {\em $\omega_{1}$-multiplicative}, if 
(\ref{abarcan}) holds for any countable subset $S\subseteq A$. 
A modal algebra is called a {\em CKL-algebra}, if 
the following hold for each $x\in A$:
\begin{enumerate}
\item
$\acko x\leq\aeko\acko x$;
\item
$\acko x=\bigsqcap_{n\in\omega} \aeko^{n} x$. 
\end{enumerate}
We write $\ackl$ for the class of all CKL-algebras. 
\end{definition}

We often identify formulas of modal logic with terms of modal algebras, 
when no confusion arises. 
Let $X$ be a set of variables. 
Formally, the set of terms of modal algebras 
over $X$ is the smallest set $T(X)$ which satisfies the following conditions:
\begin{enumerate}
\item
$0$, $1$, and any $x\in X$ are in $T(X)$;
\item
if $a$ is in $T(X)$, then $(-a)\in T(X)$;
\item
if $a$ and $b$ are in $T(X)$, then $(a\sqcup b)$ and $(a\sqcap b)$ are in $T(X)$;
\item
if $a$ is in $T(X)$, then $(\mdl a)\in T(X)$ for each 
modal operator $\mdl=\ako_{i}\  (i\in\modalop),\  \acko$. 
\end{enumerate}

\begin{definition}
An {algebraic model} is a pair $\langle A,v\rangle$, where  
$A$ 
is a modal algebra and  
$v$ is a mapping, 
which is called a {\em valuation} on $A$, 
from the set $\propvar$ of propositional variables to $A$.
For each valuation $v$ on $A$, 
the domain $\propvar$ is extended to $\cformulas$ in the following way: 
\begin{enumerate}
\item
$v(\bot)=0$, 
$v(\top)=1$; 
\item 
$v(\phi\lor\psi)=v(\phi)\sqcup v(\psi)$, $v(\phi\land\psi)=v(\phi)\sqcap v(\psi)$;
\item 
$v(\neg\phi)=
-v(\phi)$;
\item 
$v(\ko_{i}\phi)=\ako_{i} v(\phi)$ ($i\in\modalop$), $v(\cko\phi)=\acko v(\phi)$. 
\end{enumerate}
\end{definition}

In this paper, we assume a Kripke frame is equipped with 
relations 
$R_{\ko_{i}}$ ($i\in\modalop$) and $R_{\cko}$, 
unless otherwise noted.

\begin{definition}
A Kripke frame is a structure 
$F=\langle W,R_{\ko_{i}} (i\in\modalop), R_{\cko}\rangle$, 
where $W$ is a non-empty set and 
$R_{\ko_{i}}$ ($i\in\modalop$) and $R_{\cko}$
are binary relations 
on $W$. 
For any relation $R$ in $F$ and any $x\in W$, 
we write $R(x)$ for the set $\{y\in W\mid (x,y)\in R\}$. 
A Kripke frame is called 
a {\em CKL-frame}, if 
$R_{\cko}=\bigcup_{n\in\omega}\left(R_{\eko}\right)^{n}$,  
where
$
R_{\eko}=\bigcup_{i\in\modalop}R_{\ko_{i}}.
$
That is, $R_{\cko}$ is the reflexive and transitive closure of 
$R_{\eko}$.  
We write $\fckl$ for the class of all CKL-frames. 
\end{definition}

\begin{definition}
A {\em Kripke model} is a pair $\langle F,v\rangle$, where  
$F=\langle W,\{R_{\ko_{i}}\}_{i\in\modalop},R_{\cko}\rangle$ 
is a Kripke frame and  
$v$ is a mapping, 
which is called a {\em valuation} on $F$, 
from the set $\propvar$ of propositional variables to $\power(W)$.
For each valuation $v$ on $F$, 
the domain $\propvar$ is extended to $\cformulas$ in the following way: 
\begin{enumerate}
\item
$v(\bot)=\emptyset$, 
$v(\top)=W$; 
\item 
$v(\phi\lor\psi)=v(\phi)\cup v(\psi)$, $v(\phi\land\psi)=v(\phi)\cap v(\psi)$;
\item 
$v(\neg\phi)=
W\setminus v(\phi)$;
\item 
$v(\Box\phi)=\{w\in W\mid R_{\Box}(w)\subseteq v(\phi))\}$,  
for each modal operator $\Box$. 
\end{enumerate}
\end{definition}

\begin{definition}
Let $A$ be a modal algebra.  
A formula $\phi$ is said to be {\em valid} in $A$
($A\vl\phi$, in symbol), 
if $v(\phi)=1$ for any valuation 
$v$ on $A$. 
Let $C$ be a class of modal algebras. 
A formula $\phi$ is said to be {\em valid} in $C$
if $A\vl\phi$ for every $A\in C$. 
A set $\gm$ of formulas is said to  be {\em valid} in $C$
($C\vl\gm$, in symbol), 
if $C\vl\gamma$ for every $\gamma\in\gm$.  
The corresponding relations between Kripke frames and formulas are defined 
in the same way. 
\end{definition}

\begin{definition}
Let $\gm$ be a set of formulas. 
We write 
$\defalgebra(\gm)$ (resp. $\defframe(\gm)$) for the class of 
modal algebras 
(resp. Kripke frames) in which $\gm$ is valid. 
Let $C$ be a class of modal algebras or a class of Kripke frames. We write 
$\deflogic(C)$ for the set of formulas which are valid in $C$. 
\end{definition}

\begin{definition}
Let $C$ be a class of modal algebras (resp. Kripke frames). 
We say that $C$ is {\em modally definable}, if there exists a set 
$S$ of formulas of the common knowledge logic such that 
$C=\defalgebra(S)$ (resp. $C=\defframe(S)$). 
\end{definition}

Let $C$ be a class of modal algebras (resp. Kripke frames). 
It is well known that the following holds: 
\begin{equation}\label{definability}
\text{
$C$ is modally definable 
$\eq$ 
$C=\defalgebra(\deflogic(C))$ 
(resp. $C=\defframe(\deflogic(C))$)
}.
\end{equation}

Let $F=\langle W,R_{\ko_{i}} (i\in\modalop), R_{\cko}\rangle$ be 
a Kripke frame. 
It is well known that 
$$
F^{+}
=
\langle
\power(W),\cup,\cap,W\setminus, 
\mdl_{R_{\ko_{i}}} (i\in\modalop),\mdl_{R_{\cko}},\emptyset,W
\rangle
$$ 
is a complete and completely multiplicative modal algebra, where 
$$
\mdl_{R} S=\{w\in W\mid R(w)\subseteq S\}
$$
for each binary relation $R$ in $F$ and $S\subseteq W$. 
It is also well known that 
\begin{equation}\label{dualalgebra}
F\vl\phi \text{ $\eq$ } F^{+}\vl\phi
\end{equation}
holds, for any formula $\phi$
(\cite{jns-trs51,jns-trs52,thm75}, see also \cite{blc-rjk-vnm01}).

\begin{theorem}\label{cklduality}
Let $F$ be a Kripke frame. 
Then, 
$F$ is a CKL-frame, if and only if $F^{+}$ is a 
CKL-algebra. 
\end{theorem}

\begin{proof}
Let $F=\langle W,R_{\ko_{i}} (i\in\modalop), R_{\cko}\rangle$ be 
a Kripke frame. 
First, suppose that $F$ is a CKL-frame. 
Then, $R_{\cko}=\bigcup_{n\in\omega}\left(R_{\eko}\right)^{n}$. 
Take any $X\subseteq W$. 
Then, for any $w\in W$, 
\begin{align*}
w\in \Box_{R_{\cko}}X
&\eq
R_{\cko}(w)\subseteq X\\
&\eq
\left(\bigcup_{n\in\omega}\left(R_{\eko}\right)^{n}\right)(w)\subseteq X\\
&\eq
\forall n\in\omega\left(\left(R_{\eko}\right)^{n}(w)\subseteq X\right)\\
&\eq
w\in \bigcap_{n\in\omega}\left(\Box_{R_{\eko}}\right)^{n}X. 
\end{align*}
Next, suppose that 
$F^{+}$ is a CKL-algebra. Take any $x\in W$. 
We show that 
$R_{\cko}(x)=\left(\bigcup_{n\in\omega}\left(R_{\eko}\right)^{n}\right)(x)$. 
For any $w\in W$, 
\begin{align*}
w\not\in R_{\cko}(x)
&\eq
x\in \Box_{R_{\cko}}\left(W\setminus\{w\}\right)\\
&\eq
x\in \bigcap_{n\in\omega}\left(\Box_{R_{\eko}}\right)^{n}\left(W\setminus\{w\}\right)\\
&\eq
\forall n\in\omega\left(w\not\in \left(R_{\eko}\right)^{n}(x)\right)\\
&\eq
w\not\in\left(\bigcup_{n\in\omega}\left(R_{\eko}\right)^{n}\right)(x). 
\end{align*}
\end{proof}

\section{Proof systems for common knowledge logic}\label{sec:ps}

In this section, we define two equivalent proof systems $\psmh$ and 
$\psomega$ for 
the common knowledge logic, initially introduced 
by Meyer and van der Hoek \cite{mey-vdh95} and 
Kaneko-Nagashima-Suzuki-Tanaka \cite{knk-ngs-szk-tnk}, respectively.

We first define the system $\psmh$, which is originally given 
by Meyer and van der Hoek (the system $\mathbf{KEC}_{\mathbf{(m)}}$ of \cite{mey-vdh95}, 
where $m$ is the number of the knowledge operators $\ko_{1},\ldots,\ko_{m}$).

\begin{definition}\label{defpsax}
The proof system $\psmh$ for the common knowledge logic 
consists of the following axiom schemas (1)-(5) and inference rules (6)-(8): 
\begin{enumerate}
\item
all tautologies;

\item
$\mdl(\phi\supset\psi)\supset(\mdl\phi\supset\mdl\psi)$, for 
each modal operator $\mdl$;

\item
$\cko\phi\supset\phi$;  

\item\label{ckobarcan}
$\cko\phi\supset\eko\cko\phi$;

\item\label{ckoind}
$
\cko(\phi\supset\eko\phi)\supset(\phi\supset\cko\phi);
$
\item
modus ponens;

\item
uniform substitution rule;

\item
necessitation rule for each modal operator. 
\end{enumerate}
\end{definition}

Next, we define the system $\psomega$, which includes an $\omega$-rule.  
It is originally introduced by Kaneko-Nagashima-Suzuki-Tanaka 
(the system $\mathrm{CY}$ in \cite{knk-ngs-szk-tnk}), 
and is a Hilbert-style translation of the sequent system 
$\mathrm{CK}$ given in \cite{tnk01}. 

\begin{definition}\label{defpscko}
The proof system $\psomega$ 
consists of 
axiom schemas (1)-(4) and inference rules (6)-(8) in Definition \ref{defpsax},  
and the following $\omega$-rule: 
for any $k\in\omega$ and 
any modal operators $\mdl_{1},\ldots,\mdl_{k}$, 
\begin{equation}\label{cright}
\dfrac{
\gamma
\supset
\mdl_{1}(
\phi_{1}\supset
\mdl_{2}(
\phi_{2}\supset
\cdots\supset
\mdl_{k}(
\phi_{k}\supset
\eko^{n}\phi
)
\cdots
)
)
\text{\rm \hspace{10pt}($n\in\omega$) }
}
{
\gamma
\supset
\mdl_{1}(
\phi_{1}\supset
\mdl_{2}(
\phi_{2}\supset
\cdots\supset
\mdl_{k}(
\phi_{k}\supset
\cko\phi
)
\cdots
)
)
}.
\end{equation}
\end{definition}

The set of premises of the inference rule (\ref{cright}) is countable. 
When $k=0$, 
(\ref{cright}) means that 
\begin{equation}\label{cright1}
\dfrac
{\gamma\supset\eko^{n}\phi\text{\rm \hspace{10pt}($n\in\omega$) }}
{\gamma\supset\cko\phi}. 
\end{equation}

Both $\psmh$ and $\psomega$ are 
sound and the complete
with respect to $\fckl$, where the former is proved in \cite{mey-vdh95},  
and latter in \cite{tnk01, knk-ngs-szk-tnk}. 
Hence, $\psmh$ and $\psomega$ are equivalent. 
We write $\ckl$ for the set of formulas that are derivable in these 
proof systems. 
It is obvious from the Kripke completeness that the proof systems $\psmh$ and $\psomega$
are sound and complete with respect to the class $\ackl$ of 
CKL-algebras. 
Hence, we have 
\begin{equation}\label{soundnesscompleteness}
\lckl=\deflogic(\ackl).
\end{equation}

\section{Models for $\lckl$}\label{sec:models}

In this section, we show 
that the class $\fckl$ of CKL-frames is
modally definable, while 
the class $\ackl$ of CKL-algebras is not.   
First, we prove that 
$\fckl=\defframe(\lckl)$. 
Next, we introduce a class $\amh$ of MH-algebras, and 
show that 
$\amh=\defalgebra(\ckl)$.  
Then, we prove that 
$\ackl$ does not form a variety. 
Consequently, it follows that 
$\ackl$ is not modally definable. 
Finally, we show that any free MH-algebra is a CKL-algebra.

\begin{definition}\label{defmhalg}
A modal algebra is called a {\em MH-algebra}, if 
the following hold for each $x\in A$: 
\begin{enumerate}
\item\label{mh1}
$\acko x\leq x$;
\item\label{mh2}
$\acko x\leq \aeko\acko x$;
\item\label{mhalgax}
$\acko(x\pscom \aeko x)\leq x\pscom \acko x$.
\end{enumerate}
We write $\amh$ for the class of all MH-algebras. 
\end{definition}

It is clear that the system $\psmh$ is sound and complete with respect 
to $\amh$. Hence, 
\begin{equation}\label{amhcompleteness}
\ckl=\deflogic(\amh).
\end{equation}

\begin{theorem}\label{algcklmh}
The class of MH-algebras is  
the class of modal algebras in which $\ckl$ is valid. That is,  
\begin{equation}\label{amhalgckl}
\amh=\defalgebra(\ckl). 
\end{equation}
\end{theorem}

\begin{proof}
It is trivial from the definition that 
$\amh$ is equationally definable.  
Since we can identify a formula $\phi\equiv\psi$ of the common knowledge logic 
with the equation $\phi=\psi$ of modal algebras, 
$\amh$ is modally definable. 
Hence, by (\ref{definability}) and (\ref{amhcompleteness}), 
\begin{equation*}
\amh=\defalgebra(\deflogic(\amh))=\defalgebra(\ckl). 
\end{equation*}
\end{proof}

\begin{theorem}\label{cklmh}
If $A$ is a CKL-algebra, then it is an MH-algebra.  
\end{theorem}

\begin{proof}
Suppose $A$ is a CKL-algebra. Take any $x\in A$. 
We show that for any $n\in\omega$
\begin{equation}\label{mh-ax}
x\land\acko(x\pscom\aeko x)\leq\eko^{n}x 
\end{equation}
by induction on $n\in\omega$. 
The case $n=0$ is trivial. 
Suppose that 
\begin{equation}\label{indhyp}
x\land\acko(x\pscom\aeko x)\leq \eko^{k}x. 
\end{equation}
Since $A$ is a CKL-algebra, 
\begin{equation}\label{k+1}
x\land\acko(x\pscom\aeko x)
=
x\land\bigsqcap_{n\in\omega}\eko^{n}(x\pscom\aeko x)
\leq
\eko^{k}(x\pscom\aeko x)
\leq
\eko^{k}x\pscom\eko^{k+1} x.
\end{equation}
By (\ref{indhyp}) and (\ref{k+1}), 
$$
x\land\cko(x\pscom\aeko x)
\leq
\eko^{k+1} x.
$$
Hence, (\ref{mh-ax}) holds for any $n\in\omega$. 
\end{proof}

\begin{theorem}\label{mhckl}
An $\omega_{1}$-complete and $\omega_{1}$-multiplicative MH-algebra 
is a CKL-algebra. 
\end{theorem}

\begin{proof}
Suppose that $A$ is an $\omega_{1}$-complete and $\omega_{1}$-multiplicative MH-algebra. 
Take any $x\in A$. 
Since $A$ is $\omega_{1}$-complete, $\bigsqcap_{n\in\omega}\aeko^{n}x\in A$. 
We show that $\acko x=\bigsqcap_{n\in\omega}\aeko^{n}x$. 
Let $z=\bigsqcap_{n\in\omega}\aeko^{n}x$. 
By $\omega_{1}$-multiplicativity, 
$$
\aeko z
=
\aeko\bigsqcap_{n\in\omega}\aeko^{n}x
=
\bigsqcap_{n\in\omega}\aeko^{n+1}x
\geq
\bigsqcap_{n\in\omega}\aeko^{n}x
=
z. 
$$
Hence, $z\pscom \aeko z=1$. 
By (\ref{mhalgax}) of Definition \ref{defmhalg}, 
$1\leq z\pscom\acko z$, which means $z\leq\acko z$. 
Therefore, 
$$
\bigsqcap_{n\in\omega}\aeko^{n}x
=
z
\leq
\acko z
=
\acko\bigsqcap_{n\in\omega}\aeko^{n}x
\leq
\acko x. 
$$
By (\ref{mh1}) and (\ref{mh2}) of Definition \ref{defmhalg}, for any $n\in\omega$,  
$$
\acko x\leq \aeko^{n}x. 
$$
Hence, 
$\acko x\leq \bigsqcap_{n\in\omega}\aeko^{n}x$. 
\end{proof}

\begin{corollary}
A Kripke frame $F$ is a CKL-frame if and only if $F\vl\lckl$. 
Hence, the class of CKL-frames is modally definable. 
\end{corollary}

\begin{proof}
By (\ref{dualalgebra}) and (\ref{amhalgckl}), and 
Theorem \ref{cklduality}, \ref{cklmh}, and \ref{mhckl}, 
\begin{align*}
F\in\fckl
&\eq 
F^{+}\in\ackl\\
&\eq 
F^{+}\in\amh\\
&\eq
F^{+}\in\defalgebra(\ckl)\\
&\eq 
F^{+}\vl\ckl\\
&\eq 
F\vl\ckl. 
\end{align*}
\end{proof}

We now demonstrate that $\ackl\subsetneqq\amh$.

\begin{lemma}\label{ultrapower}
Let $V$ be a subvariety of $\amh$.  
Suppose that, for each $n\in\omega$, there exists $A_{n}\in V$ and 
$a_{n}\in A$ such that  
\begin{equation}\label{ackolneqq}
\acko a_{n}\lneqq \bigsqcap_{i=0}^{n}\aeko^{i}a_{n}. 
\end{equation}
Then the ultraproduct $A=\prod_{n\in\omega} A_{n}/U$ is an element of  $V\setminus\ackl$, 
where $U$ is the set of all cofinite subsets of $\power(\omega)$. 
\end{lemma}

\begin{proof}
Let $a=(a_{n})_{n\in\omega}\in \prod_{n\in\omega}A_{n}$. 
Let $b_{n}=\bigsqcap_{i=0}^{n}\aeko^{i}a_{n}\in A_{n}$ for any $n\in\omega$ and 
$b=(b_{n})_{n\in\omega}\in \prod_{n\in\omega}A_{n}$. Then, for any $n\in\omega$, 
\begin{equation}\label{bleqena}
b/U\leq (\aeko^{n}a)/U, 
\end{equation}
since 
$$
\{m\in\omega\mid b_{m}\leq \aeko^{n}a_{m}\}
\supseteq
\{m\in\omega\mid m\geq n\}
\in
U. 
$$
On the other hand, 
\begin{equation}\label{clneqqb}
(\acko a)/U\lneqq b/U, 
\end{equation}
since 
$$
\{n\in\omega\mid \acko a_{n} \lneqq b_{n}\}
=
\omega
\in
U,  
$$
by \eqref{ackolneqq}. Hence, 
$\acko\left(a/U\right)$ is not the meet of 
$
\{\aeko^{n}\left(a/U\right)\mid n\in\omega\} 
$
by \eqref{bleqena} and \eqref{clneqqb}. 
\end{proof}

\begin{theorem}\label{ckl-mh}
$\ackl\subsetneqq\amh$. 
\end{theorem}

\begin{proof}
By Lemma \ref{ultrapower}, it is sufficient to show the existence of a 
set $\{A_{n}\}_{n\in\omega}$ of MH-algebras that satisfies \eqref{ackolneqq}. 
Define a Kripke frame $F=\langle W,R_{\ko_{0}},R_{\ko_{1}},R_{\cko}\rangle$ by 
$W=\omega\cup\{\infty\}$, 
$$
R_{\ko_{0}}=\mathrm{Eq}\bigl(
\{(0,\infty)\}
\cup
\{(2n,2n+1)\mid n\in\infty\}
\bigr),
$$
$$
R_{\ko_{1}}=\mathrm{Eq}\bigl(
\{(0,\infty)\}
\cup
\{(2n+1,2n+2)\mid n\in\infty\}
\bigr),
$$
where $\mathrm{Eq}(R)$ denotes the equivalence relation generated by $R$, 
and 
$$
R_{\cko}=\bigcup_{n\in\omega}(R_{\eko})^{n}, 
$$
where $R_{\eko}=R_{\ko_{0}}\cup R_{\ko_{1}}$
(Figure \ref{kripkeframe}). 
Then, $F\in\fckl$. Hence, $F^{+}\in\amh$ by Theorem \ref{cklduality}
and Theorem \ref{cklmh}. 
In $F^{+}$, 
$$
\bigcap_{i=0}^{n}\aeko^{i}\omega=\{m\in\omega \mid m\geq n\}\not=\emptyset=\acko\omega,  
$$
for any $n\in\omega$.
Define $A_{n}=F^{+}$ for each $n\in\omega$. 
\end{proof}


\begin{figure}[h]
\begin{tikzcd}
\infty
\arrow[d, shift right=2pt]\arrow[d, dashed, shift left=5pt]
&&&&&
\\
0
\arrow[u, shift left=5pt]\arrow[u, dashed, shift right=2pt]
\arrow[r, shift left=2pt]
&
1\arrow[l, shift left=2pt, "\ko_{0}"]\arrow[r, dashed, shift left=2pt]
&
2\arrow[l, dashed, shift left=2pt, "\ko_{1}"]\arrow[r, shift left=2pt]
&
3\arrow[l, shift left=2pt, "\ko_{0}"]\arrow[r, dashed, shift left=2pt]
&
4\arrow[l, dashed, shift left=2pt, "\ko_{1}"]\cdots
&
\cdots
\end{tikzcd}
\caption{Kripke frame $F$}\label{kripkeframe}
\end{figure}


As $R_{\ko_{0}}$ and $R_{\ko_{1}}$ are equivalence relations, 
Theorem \ref{ckl-mh} holds for subvarieties of $\amh$ corresponding to 
$\logic{K}$, $\logic{D}$, $\logic{T}$, $\logic{4}$, $\logic{5}$ 
and their combinations.

\begin{corollary}
The class of CKL-algebras is not modally definable. Hence, it is not a variety. 
\end{corollary}

\begin{proof}
By (\ref{soundnesscompleteness}), (\ref{amhalgckl}), and Theorem \ref{ckl-mh}, 
$$
\ackl\subsetneqq \amh=\defalgebra(\ckl)=\defalgebra(\deflogic(\ackl)). 
$$
Therefore, $\ackl$ is not modally definable by (\ref{definability}). 
Hence, it is not equationally 
definable. 
Therefore, it is not a variety. 
\end{proof}

Below, we present a concrete example of a modal algebra $A$ in $\amh\setminus\ackl$. 
Let $x\in 2^{\omega}$. 
We say that $x$ is {\em finite}, 
if the cardinality of the set $\{n\mid x(n)=1\}$ is finite,  
and say that $x$ is {\em cofinite}, if the cardinality of the set 
$\{n\mid x(n)=0\}$ is finite. 
The constant function from $\omega$ to  $2$ that takes the value $0$ (resp. $1$)
is simply denoted as $0$ (resp. $1$), if there is no confusion. 
By considering $2$ as a two-valued Boolean algebra,  
$2^{\omega}$ is a complete Boolean algebra.  
Define $S\subseteq 2^{\omega}$ by 
$$
S=
\{x\mid 
\text{$x$ is finite or cofinite}
\}.
$$
Then, $S$ is a sub-Boolean algebra of $2^{\omega}$. 
For each cofinite $x\in S$, define $k(x)\in\omega$ as follows: 
$$
k(x)=\min\{i\mid 
\text{$i$ is an even number and for each even number $j\geq i$, 
$x(j)=1$}
\}
$$ 
Since $x$ is cofinite, $k(x)$ is well-defined. 
Suppose that $|\modalop|=N$ and  
define modal operators $\ako_{n}$ ($n\in\modalop$) on $S$ as follows
(see Figure \ref{knx}): 
\begin{enumerate}
\item
if $x$ is finite, then $\ako_{n}(x)=0$;
\item
if $x=1$, then $\ako_{n}(x)=1$;
\item
if $x$ is cofinite, $x\not=1$, and $n\not\equiv k(x)$ (mod $N$), then 
$$
\ako_{n}(x)(i)
=
\begin{cases}
0
& \text{if $i$ is even and $i< k(x)$}\\
x(i) 
& \text{if $i$ is odd or $i\geq k(x)$};
\end{cases}
$$
\item
if $x$ is cofinite, $x\not= 1$, and $n\equiv k(x)$ (mod $N$), then 
$$
\ako_{n}(x)(i)
=
\begin{cases}
0
& \text{if $i$ is even and $i\leq k(x)$}\\
x(i) 
& \text{if $i$ is odd or $i> k(x)$}.\end{cases}
$$
\end{enumerate}

\noindent
Define modal operator $\acko$ on $S$ by 
$$
\acko x= 
\begin{cases}
1
& \text{if $x=1$}\\
0
& \text{otherwise}. 
\end{cases}
$$
It is easy to see that for any cofinite $x\in S$ that is not $1$,  
$$
\aeko(x)(i)
=
\bigsqcap_{n\in\modalop}\ako_{n}(x)(i)
=
\begin{cases}
0
& \text{if $i$ is even and $i\leq k(x)$}\\
x(i) 
& \text{if $i$ is odd or $i> k(x)$}.
\end{cases}
$$

\begin{figure}
$$
\begin{array}{c|cccccccc}
i &   0 & 2    & \cdots         & k(x)-2 & k(x) & k(x)+2 & k(x)+4 & \cdots \\
\hline
x   &\ast & \ast & \ast\cdots\ast & 0      & 1    & 1      & 1      &  1\cdots \\
\ako_{n}x\text{ ($n\not\equiv k(x)$)}
    &0    & 0    & 0\cdots0       & 0      & 1    & 1      & 1      &  1\cdots \\
\ako_{n}x\text{ ($n\equiv k(x)$)}
    &0    & 0    & 0\cdots0       & 0      & 0    & 1      & 1      &  1\cdots \\
\aeko x
    &0    & 0    & 0\cdots0       & 0      & 0    & 1      & 1      &  1\cdots 
\end{array}
$$
\caption{$\ako_{n} x(i)$ and $\aeko x(i)$ for cofinite $x\not=1$ and even $i$}\label{knx}
\end{figure}

\noindent
We first check that $S$ is a modal algebra. 
By definition, $\ako_{n} 1=1$ ($n\in\modalop$) and $\acko 1=1$. 
It is straightforward to show that 
$\acko (x\sqcap y)=\acko x\sqcap \acko y$. 
We check that $\ako_{n} (x\sqcap y)=\ako_{n}x\sqcap \ako_{n}y$ holds. 
The cases that $x=1$ or $y=1$, and $x$ is finite or $y$ is finite are straightforward. 
Suppose that $x$ and $y$ are cofinite, $x\not=1$, and $y\not=1$. 
It is easy to see that  
$\ako_{n} (x\sqcap y)(i)=(\ako_{n}x\sqcap \ako_{n}y)(i)$, for each odd number 
$i\in\omega$.  
Without loss of generality, we may assume that $k(x)\leq k(y)$. 
Then, $k(x\sqcap y)=k(y)$. 
Suppose that  $n\equiv k(y)$ (mod $N$).  Then, for each even number $i\in\omega$, 
$$
\ako_{n} (x\sqcap y)(i)=0
\eq
i\leq k(x\sqcap y)
\eq
i\leq k(y)
\eq 
(\ako_{n}x\sqcap \ako_{n}y)(i)=0. 
$$
The case $n\not\equiv k(y)$ (mod $N$) is shown in the same way. 
Hence, $S$ is a modal algebra. 
Next, we show that $S$ is an MH-algebra. 
It is easy to see that $\acko x\leq x$ and $\acko x\leq \aeko\acko x$ 
hold. We show that $\acko(x\pscom \aeko x)\leq x\pscom \acko x$ holds.
The cases that $x=0$ and $x=1$ are straightforward. Suppose not. 
Then, $x\not\leq\aeko x$. Therefore, $x\pscom\aeko x\not=1$. 
Hence, 
$$
\acko(x\pscom\aeko x)=0
\leq
x\pscom \acko x. 
$$ 
Finally, we show that $S$ is not a CKL-algebra. Let $a\in S$ be 
$$
a(i)
=
\begin{cases}
0 & (i=0)\\
1 & (i\not=0)
\end{cases}. 
$$
Then, in $2^{\omega}$, 
$$
\bigsqcap_{n\in\omega}\aeko^{n}a(i)
=
\begin{cases}
0 & (\text{$i$ is even})\\
1 & (\text{$i$ is odd})
\end{cases},  
$$
but the greatest upper bound of the set $\{\aeko^{n}a\mid n\in\omega\}$ does 
not exist in $S$.

In the last part of the section, we show that the free MH-algebras are 
CKL-algebras. 
Let $\psnon$ be a proof system and  
let $\sim_{\psnon}$ be the binary relation on $\cformulas$ defined such that 
$\phi\sim_{\psnon}\psi$ if and only if $\phi\equiv\psi$ is derivable 
in $\psnon$, 
for each formulas $\phi$ and $\psi$. 
We write $\linnon$ for the Lindenbaum-Tarski algebra
$\cformulas/\text{$\sim_{\psnon}$}$.
For each formula $\phi$, we write $|\phi|_{\psnon}$ for the equivalence 
class of $\phi$ in $\cformulas/\text{$\sim_{\psnon}$}$.

Let $X$ be a set of variables and 
$\sim_{\mathrm{MH}}$ be the binary relation on the set $T(X)$ of all terms 
over $X$ 
defined such that 
for each $t_{1}$ and $t_{2}$ in $T(X)$, $t_{1}\sim_{\mathrm{MH}}t_{2}$ if and 
only if $t_{1}=t_{2}$ holds in every MH-algebra. 
It is known that $\sim_{\mathrm{MH}}$ is a congruence relation, and 
that the free MH-algebra over $X$ is $T(X)/\sim_{\mathrm{MH}}$. 
We denote the free MH-algebra $T(X)/\sim_{\mathrm{MH}}$ over $X$ 
by $\freemh(X)$. 
For each $t\in T(X)$, we denote 
the equivalence class of $t$ in 
$T(X)/\sim_{\mathrm{MH}}$ 
by $|t|_{\mathrm{MH}}$. 
We define the free CKL-algebra 
$\freeomega(X)$ over $X$ and 
symbols $|t|_{\mathrm{CKL}}$ 
in the same manner.

\begin{theorem}\label{freemh}
For each set $X$, the free 
MH-algebra $\freemh(X)$ over $X$ is a CKL-algebra. 
\end{theorem}

\begin{proof}
Let $\phi$ be a formula of the common knowledge logic. 
By (3) and (4) of Definition \ref{defpsax} and (\ref{cright1}),  
\begin{equation}
|\cko\phi|_{\psomega}=
\bigsqcap_{n\in\omega}|\eko^{n}\phi|_{\psomega}
\end{equation}
holds in the Lindenbaum-Tarski algebra $\linomega$ of $\psomega$.  
Therefore, $\linomega$ is a CKL-algebra. 
Since $\psmh$ and $\psomega$ are equivalent, 
$\linmh=\linomega$. 
Suppose that $X$ is countable. 
Then, we can recursively define a bijection
from the set of formulas of the common knowledge logic to the 
set of terms of modal algebras. 
It is straightforward to see that $\phi\equiv\psi$ is derivable in $\psomega$ if 
and only if $\phi=\psi$ holds in $\ackl$, and the same 
relation holds between $\psmh$ and $\amh(X)$. 
Therefore, 
\begin{equation*}\label{freecountable}
\freeomega(X)\cong\linomega=\linmh\cong\freemh(X). 
\end{equation*}
Hence, $\freemh(X)$ is a CKL-algebra.  
Take any set $Y$ and 
any $|t_{1}|_{\mathrm{MH}}$ and $|t_{2}|_{\mathrm{MH}}$ in $\freemh(Y)$. 
Let $y_{1},\ldots,y_{n}$ be the list of variables in $Y$ which occur in 
$t_{1}$ or $t_{2}$,  
and let $X$ be a countable set of variables which includes all $y_{1},\ldots,y_{n}$. 
Then, it is known that 
\begin{align*}
|t_{1}|_{\mathrm{MH}}=|t_{2}|_{\mathrm{MH}} \text{ in $\freemh(Y)$ }
\eq
|t_{1}|_{\mathrm{MH}}=|t_{2}|_{\mathrm{MH}} \text{ in $\freemh(X)$ }
\end{align*}
holds (see, e.g., Theorem 11.4 of \cite{brr-snk80}). 
Now, let $|s|_{\mathrm{MH}}$ and $|t|_{\mathrm{MH}}$ be in $\freemh(Y)$. 
Take any $X$ which includes all variables in $s$ and $t$. 
Then, all variables that occur in $\acko s$, $\aeko^{n} s$ ($n\in\omega$), 
and $t$ are in $X$. 
Since $\freemh(X)$ is a CKL-algebra, 
\begin{align*}
& 
\text{for any $n\in\omega$, }
|t|_{\mathrm{MH}}\leq \aeko^{n}|s|_{\mathrm{MH}} \text{ in $\freemh(Y)$}\\
\eq&
\text{for any $n\in\omega$, }
|t|_{\mathrm{MH}}\leq \aeko^{n}|s|_{\mathrm{MH}} \text{ in $\freemh(X)$}\\
\thn&
|t|_{\mathrm{MH}}\leq \acko|s|_{\mathrm{MH}} \text{ in $\freemh(X)$}\\
\eq&
|t|_{\mathrm{MH}}\leq \acko|s|_{\mathrm{MH}} \text{ in $\freemh(Y)$}. 
\end{align*}
It is clear that for any $n\in\omega$, 
$\acko|s|_{\mathrm{MH}}\leq \aeko^{n}|t|_{\mathrm{MH}}$ holds in $\freemh(Y)$. 
Hence, 
$$
\acko|s|_{\mathrm{MH}}=\bigsqcap_{n\in\omega}\aeko^{n}|s|_{\mathrm{MH}}
$$
holds in $\freemh(Y)$. 
\end{proof}

It is known that an equation holds in a class of algebras 
if and only if it holds in the free algebra over a countable 
set of variables (see, e.g., corollary 11.5 of \cite{brr-snk80}). 
However, there exists an MH-algebra which 
does not satisfy the equation 
$
\acko x=\bigsqcap_{n\in\omega}\aeko^{n}x
$, 
while the free MH-algebra over a countable set $X$ satisfies it. 
Of course, this is not a paradox, as
$\bigsqcap_{n\in\omega}\aeko^{n}x$ is not a term of MH-algebras.

\section{Infinitary common knowledge logic}\label{sec:inflogic}

In this section, we discuss the infinitary common knowledge logic. 
We define the infinitary extensions 
$\ipsmh$ and $\ipsomega$ of $\psmh$ and $\psomega$, respectively.  
The system $\ipsomega$ is originally introduced by Kaneko-Nagashima-Suzuki-Tanaka 
\cite{knk-ngs-szk-tnk}. 
It is proved in \cite{knk-ngs-szk-tnk} that  
$\ipsomega$ is sound and complete with respect to $\fckl$. 
We show that $\ipsmh$ and $\ipsomega$ are equivalent.  
Consequently, it follows that $\ipsmh$ is also 
sound and complete with respect to $\fckl$.

We define syntax and semantics for the infinitary common knowledge logic. 
First, we define syntax. 
We extend the language by adding  two logical connectives $\ior$ and $\iand$, 
which denote countable disjunction and conjunction, respectively. 
Then, we define the set $\icformulas$ of formulas of the infinitary common knowledge logic 
as the least extension of $\cformulas$ 
which satisfies the following: 
if $\gm$ is a countable set of formulas of $\icformulas$, then 
$\ior \gm$ and $\iand \gm$ are in $\icformulas$. 
Next, we define semantics. 
An algebraic model for the infinitary common knowledge logic 
is a pair 
$\langle A,v\rangle$, where  
$A$ 
is an $\omega_{1}$-complete and $\omega_{1}$-multiplicative modal algebra and  
$v$ is a mapping 
from the set $\propvar$ of propositional variables to $A$.
For each valuation $v$ on $A$ and each countable set $\gm$ of formulas, 
we define 
$$
v\left(\ior \gm\right)=\bigsqcup_{\gamma\in\gm}v(\gamma),\   
v\left(\iand \gm\right)=\bigsqcap_{\gamma\in\gm}v(\gamma). 
$$
A Kripke model for the infinitary common knowledge logic 
is a pair $\langle F,v\rangle$, where  
$F=\langle W,\{R_{\ko_{i}}\}_{i\in\modalop},R_{\cko}\rangle$ 
is a Kripke frame and  
$v$ is a mapping 
from the set $\propvar$ of propositional variables to $\power(W)$.
For each valuation $v$ on $F$ and each countable set $\gm$ of formulas, 
we define
$$
v\left(\ior \gm\right)=\bigcup_{\gamma\in\gm}v(\gamma),\   
v\left(\iand \gm\right)=\bigcap_{\gamma\in\gm}v(\gamma). 
$$

Now, we define $\ipsmh$ and $\ipsomega$. 
First, we define $\ipsmh$. 

\begin{definition}\label{defipsax}
The proof system $\ipsmh$ for the infinitary common knowledge logic 
consists of all axiom schemas and inference rules of $\psmh$ and the 
following axiom schemas and inference rules:  
\begin{equation}\label{infbarcan}
\iand_{\gamma\in\gm}\mdl\gm\supset\mdl\iand\gm \ \ 
\text{for each modal operator $\mdl$},
\end{equation}
\begin{equation}\label{infaxiom}
\gamma\supset\ior\gm,\ \iand\gm\supset\gamma\ \ 
\text{($|\gm|\leq\omega$, $\gamma\in\gm$)}, 
\end{equation}
\begin{equation}\label{infrule}
\dfrac
{\gamma\supset\phi\text{\rm \hspace{10pt}($\forall\gamma\in\gm$) }}
{\ior\gamma\supset\phi},\ 
\dfrac
{\phi\supset\gamma\text{\rm \hspace{10pt}($\forall\gamma\in\gm$) }}
{\phi\supset\iand\gamma}\ \ 
\text{($|\gm|\leq\omega$)}.
\end{equation}

\end{definition}

Next, we define  $\ipsomega$, which 
is originally introduced by Kaneko-Nagashima
(the system $\mathrm{GL}_{\omega}$ in \cite{knk-ngs96, knk-ngs97}). 

\begin{definition}\label{defipscko}
The proof system $\ipsomega$ 
consists of all axioms and inference rules of $\psomega$ and axioms
(\ref{infbarcan}) and (\ref{infaxiom}) 
and inference rules
(\ref{infrule}). 
\end{definition}

It is proved in \cite{knk-ngs-szk-tnk} that 
$\ipsomega$ is Kripke complete:

\begin{theorem}\label{infalgcompleteness}
(\cite{knk-ngs-szk-tnk}). 
For each formula $\phi\in\icformulas$,  
$\phi$ is provable in $\ipsomega$ if and only if 
it is valid in the class $\fckl$ of CKL-frames. 
\end{theorem}

We show the algebraic completeness of $\ipsmh$ and $\ipsomega$. 

\begin{theorem}
For each formula $\phi\in\icformulas$,  the following conditions 
are equivalent: 
\begin{enumerate}
\item
$\phi$ is provable in $\ipsomega$;
\item
$\phi$ is provable in $\ipsmh$;
\item
$\phi$ is valid in the class of $\omega_{1}$-complete and 
$\omega_{1}$-multiplicative CKL-algebras. 
\end{enumerate}
\end{theorem}

\begin{proof}
It is straightforward to show the soundness of $\ipsomega$ and $\ipsmh$ 
with respect to the class of $\omega_{1}$-complete and 
$\omega_{1}$-multiplicative CKL-algebras. 
It is obvious that the Lindenbaum-Tarski algebra of 
$\ipsomega$ is an $\omega_{1}$-complete and 
$\omega_{1}$-multiplicative CKL-algebra.  
By Theorem \ref{mhckl}, it follows that 
the Lindenbaum-Tarski algebra of $\ipsmh$ 
is also an  $\omega_{1}$-complete and 
$\omega_{1}$-multiplicative CKL-algebra. 
Hence, 
$\ipsomega$ and $\ipsmh$ are complete 
with respect to the class of $\omega_{1}$-complete and 
$\omega_{1}$-multiplicative CKL-algebras. 
\end{proof}

By Theorem \ref{infalgcompleteness}, 
$\ipsmh$ and $\ipsomega$ are equivalent. Hence, 
we have the following: 

\begin{corollary}\label{infcompleteness}
For each formula $\phi\in\icformulas$,  
$\phi$ is provable in $\ipsmh$ if and only if 
it is valid in the class $\fckl$ of CKL-frames. 
\end{corollary}

Interestingly, the Kripke completeness of $\ipsmh$ can also be shown directly, by 
using a kind of general
technique, whereas proving the Kripke completeness of $\psmh$ is not straightforward
(see \cite{{mey-vdh95}}). 
It is well-known that the J\'{o}nsson-Tarski representation theorem provides 
Kripke completeness of many types of modal logics.  
Furthermore, it is known that an extension of  the J\'{o}nsson-Tarski 
representation 
theorem, which preserves countably many infinite joins and meets, holds,  
and the Kripke completeness of various non-compact modal logics follows from this 
extension in the same way (\cite{tnk-ono98,tnknoncpt,tnk01}). 
In fact, Kripke completeness of some proof systems for the common knowledge 
logic is proved by means of this extension (\cite{tnk-ono98,tnknoncpt,tnk01}). 
When proving the Kripke completeness of the common knowledge logic 
using this extension, 
it is necessary to show that 
$\acko x=\bigsqcap_{n\in\omega}\aeko^{n}x$ 
holds in the 
Lindenbaum-Tarski algebra. 
As we have seen in Theorem \ref{freemh}, it is true that the Lindenbaum-Tarski 
algebra of $\psmh$ satisfies 
$\acko x=\bigsqcap_{n\in\omega}\aeko^{n}x$,  
but, in Theorem \ref{freemh}, 
this was shown by using the Kripke completeness of $\psmh$. 
Here, we need to show that 
$\acko x=\bigsqcap_{n\in\omega}\aeko^{n}x$ holds in the 
Lindenbaum-Tarski algebra of $\psmh$ without relying on Kripke completeness. 
However, this does not seem to be straightforward.  
On the other hand, it follows immediately from Theorem \ref{mhckl} that 
the Lindenbaum-Tarski algebra of $\ipsmh$ satisfies 
$\acko x=\bigsqcap_{n\in\omega}\aeko^{n}x$.  
Therefore, it is straightforward to show the Kripke completeness of $\ipsmh$ 
directly from the extension of J\'{o}nsson-Tarski representation.

Finally, we compare the infinitary extensions of $\psomega$ and $\psmh$ with 
their predicate extensions. 
For simplicity, we consider predicate extensions without 
function symbols and equality. 
The predicate extensions of $\psomega$ and $\psmh$ are defined 
by adding the following axiom schemas and inference rules to them, respectively: 
\begin{enumerate}
\item
$\phi[y/x]\supset\exists\phi$, 
$\forall x\phi\supset\phi[y/x]$;

\item
$
\dfrac
{\phi[y/x]\supset\psi}
{\exists x\phi\supset\psi}
$, 
$
\dfrac
{\psi\supset\phi[y/x]}
{\psi\supset\forall x\phi}
$
(where $y$ is a variable 
that does not occur in the conclusion); 

\item
Barcan formula: 
$
\barcan.
$
\end{enumerate}
As observed, both the infinitary extensions of $\psomega$ and $\psmh$ are 
Kripke complete. 
However, their predicate extensions exhibit different properties. 
It is proved in \cite{knk-ngs-szk-tnk} 
that the predicate extension of $\psomega$ is sound and complete 
with respect to the class of CKL-frames with constant domains. 
However, Wolter \cite{wlt00} proved that the set of 
formulas of the predicate common knowledge logic which is valid 
in the class of CKL-frames with constant domains is not recursively 
enumerable.  
Therefore, 
the predicate extension of $\psmh$ is not complete with respect to the 
class of CKL-frames with constant domains. 
Theorem \ref{infcompleteness} serves as an example demonstrating the 
distinct properties exhibited by infinitary extensions and predicate extensions.

\bibliographystyle{plain}
\bibliography{myref.sjis}

\newcommand{\noop}[1]{}
\begin{thebibliography}{10}

\bibitem{blc-rjk-vnm01}
Patrick Blackburn, Maarten de~Rijke, and Yde Venema.
\newblock {\em Modal Logic}.
\newblock Cambridge, third edition, 2001.

\bibitem{bch-kzn-std10}
Samuel Bucheli, Roman Kuznets, and Thomas Studer.
\newblock Two ways to common knowledge.
\newblock {\em Electronic Notes in Theoretical Computer Science}, 262:83--94,
  2010.

\bibitem{brr-snk80}
S.~Burris and H.~P. Sankappnavar.
\newblock {\em A Course in Universal Algebra}.
\newblock Springer-Verlag, 1980.

\bibitem{hlp-mss92}
Joseph~Y. Halpern and Yoram Moses.
\newblock A guide to completeness and complexity for modal logics of knowledge
  and beliefs.
\newblock {\em Artificial Intelligence}, 54:319--379, 1992.

\bibitem{jns-trs51}
Bjarni J{\'o}nsson and Alfred Tarski.
\newblock {Boolean} algebras with operators {I}.
\newblock {\em American Journal of Mathematics}, 73:891--931, 1951.

\bibitem{jns-trs52}
Bjarni J{\'o}nsson and Alfred Tarski.
\newblock {Boolean} algebras with operators {II}.
\newblock {\em American Journal of Mathematics}, 74:127--162, 1952.

\bibitem{knk-ngs96}
Mamoru Kaneko and Takashi Nagashima.
\newblock Game logic and its applications {I}.
\newblock {\em Studia Logica}, 57:325--354, 1996.

\bibitem{knk-ngs97}
Mamoru Kaneko and Takashi Nagashima.
\newblock Game logic and its applications {II}.
\newblock {\em Studia Logica}, 58:273--303, 1997.

\bibitem{knk-ngs-szk-tnk}
Mamoru Kaneko, Takashi Nagashima, Nobu-Yuki Suzuki, and Yoshihito Tanaka.
\newblock Map of common knowledge logics.
\newblock {\em Studia Logica}, 71:57--86, 2002.

\bibitem{mey-vdh95}
John-Jules~Ch. Meyer and Wiebe van~der Hoek.
\newblock {\em Epistemic Logic for AI and Computer Science}.
\newblock Cambridge Tracts in Theoretical Computer Science. Cambridge, 1995.

\bibitem{sgb94}
Krister Segerberg.
\newblock A model existence theorem in infinitary propositional modal logic.
\newblock {\em Journal of Philosophical Logic}, 23:337--367, 1994.

\bibitem{tnknoncpt}
Yoshihito Tanaka.
\newblock Model existence in non-compact modal logic.
\newblock {\em Studia Logica}, 67:61--73, 2001.

\bibitem{tnk01}
Yoshihito Tanaka.
\newblock Some proof systems for predicate common knowledge logic.
\newblock {\em Reports on Mathematical Logic}, 37:79--100, 2003.

\bibitem{tnk-ono98}
Yoshihito Tanaka and Hiroakira Ono.
\newblock The {Rasiowa-Sikorski} lemma and {Kripke} completeness of predicate
  and infinitary modal logics.
\newblock In Michael Zakharyaschev, Krister Segerberg, Maarten de~Rijke, and
  Heinrich Wansing, editors, {\em Advances in Modal Logic}, volume~2, pages
  419--437. CSLI Publication, 2000.

\bibitem{thm75}
Steven~K. Thomason.
\newblock Categories of frames for modal logic.
\newblock {\em The Journal of Symbolic Logic}, 40(3):439--442, 1975.

\bibitem{wlt00}
Frank Wolter.
\newblock First order common knowledge logics.
\newblock {\em Studia Logica}, 65:249--271, 2000.

\end{thebibliography}

\end{document}